\documentclass[12pt]{amsart}
\usepackage{a4wide}
\usepackage[utf8]{inputenc}
\usepackage[english]{babel}
\usepackage[hidelinks]{hyperref}
\usepackage{graphicx} 
\usepackage{amssymb}
\usepackage{enumitem}
\usepackage{todonotes}
\usepackage{comment}
\theoremstyle{definition}
\newtheorem{definition}{Definition}[section]

\newtheorem*{remark*}{Remark}
\newtheorem*{claim*}{Claim}
\theoremstyle{remark}

\def\str#1{\mathbf {#1}}
\def\Emb{\mathop{\mathrm{Emb}}\nolimits}

\theoremstyle{plain}
\newtheorem{theorem}{Theorem}[section]

\newtheorem{observation}[theorem]{Observation}
\newtheorem{lemma}[theorem]{Lemma}

\newtheorem{problem}[theorem]{Problem}

\DeclareFontEncoding{LS1}{}{}
\DeclareFontSubstitution{LS1}{stix2}{m}{n}
\DeclareSymbolFont{stixsym}       {LS1}{stix2scr}   {m} {n}
\DeclareSymbolFont{stixsym2}       {LS1}{stix2frak}   {m} {n}
\DeclareMathSymbol{\bigblacktriangledown}     {\mathord}{stixsym}{"D3}
\DeclareMathSymbol{\bigtriangledowns}     {\mathord}{stixsym}{"D4}
\DeclareMathSymbol{\downtriangleleftblack}     {\mathord}{stixsym2}{"D7}
\DeclareMathSymbol{\downtrianglerightblack}     {\mathord}{stixsym2}{"D8}

\def\seq{\mathrm{seq}}
\def\osc{\mathrm{osc}}
\def\int{\mathrm{int}}
\title[Oscillating Boolean algebras]{Oscillating subalgebras of the atomless countable Boolean algebra}
\author[D.\ B.]{Dana Bartošová}
\address{Department of Mathematics, University of Florida, 1400 Stadium Road, Gainesville, FL 32601}
\email{dbartosova@ufl.edu}
\author[D.\ Ch.]{David Chodounský}
\address{Department of Applied Mathematics (KAM), Charles University, Ma\-lo\-stranské~nám\v estí 25, Praha 1, Czech Republic and Institute of Mathematics of the Czech Academy of Sciences, Žitná 25, Praha 1, Czech Republic}
\email{chodounsky@math.cas.cz}
\author[B.\ C.]{Barbara Csima}
\address{Department of Pure Mathematics, University of Waterloo, 200 University Ave W, Waterloo, ON N2L 3G1, Canada}
\email{csima@uwaterloo.ca}
\author[J.\ H.]{Jan Hubička}
\address{Department of Applied Mathematics (KAM), Charles University, Ma\-lo\-stranské~nám\v estí 25, Praha 1, Czech Republic}
\email{hubicka@kam.mff.cuni.cz}
\author[M.\ K.]{Matěj Konečný}
\address{Institute of Algebra, TU Dresden, Dresden, Germany}
\email{matej.konecny@tu-dresden.de}
\author[J.\ L.-G.]{Joey Lakerdas-Gayle}
\address{Department of Pure Mathematics, University of Waterloo, 200 University Ave W, Waterloo, ON N2L 3G1, Canada}
\email{ja2999lakerdasgayle@uwaterloo.ca}
\author[S.\ T.]{Stevo Todorcevic}
\address{Department of Mathematics, University of Toronto, Canada}
\email{stevo@math.toronto.edu}
\address{Institut de Mathématiques de Jussieu, CNRS, Paris, France}
\email{stevo.todorcevic@imj-prg.fr}
\address{Matematički Institut, SANU, Belgrade, Serbia}
\email{stevo.todorcevic@sanu.ac.rs}
\author[S.\ U.]{Spencer Unger}
\address{University of Toronto, Canada}
\email{spencer.unger@utoronto.ca}
\author[A.\ Z.]{Andy Zucker}
\address{Department of Pure Mathematics, University of Waterloo, 200 University Ave W, Waterloo, ON N2L 3G1, Canada}
\email{a3zucker@uwaterloo.ca}

\begin{document}
\begin{abstract}
	We show that the big Ramsey degree of the Boolean algebra with 3 atoms in the countable atomless Boolean algebra is infinite.
\end{abstract}

\maketitle

\section{Introduction}
In this note we prove:
\begin{theorem}
	\label{thm:main}
	Let $\str{B}$ be the countable atomless Boolean algebra. Then there exists a coloring of the 3-atom subalgebras of $\str{B}$ with infinitely many colors such that
	every countable atomless subalgebra of $\str{B}$ contains a 3-atom subalgebra of every color.
\end{theorem}
This 
is motivated by the following problem formulated by Kechris, Pestov, and Todorcevic in their seminal paper connecting structural Ramsey theory and topological dynamics~\cite{Kechris2005} (and was also explicitly asked by Ma\v{s}ulovi\'c in~\cite{Masulovic2023}).
\begin{problem}[Kechris--Pestov--Todorcevic 2005~\cite{Kechris2005}, Problem 11.2]
	\label{prob:1}
	In each of the following cases, find the topological dynamics analog of a corresponding infinite Ramsey-theoretic result.
	\begin{enumerate}[label=(\roman*)]
	  \item \label{p1} The rationals with the usual ordering.
	  \item \label{p2} The random ordered graph.
	  \item \label{p3} The random $\str{K}_n$-free ordered graph ($n=3,4, \ldots$).
	  \item \label{p4} The random $\mathcal{A}$-free ordered hypergraph of type $L_0$, for any set $\mathcal A$ of irreducible finite hypergraphs of type $L_0$.
	  \item \label{p5} The ordered rational Urysohn space.
	  \item \label{p6} The $\aleph_0$-dimensional vector space over a finite field with the canonical ordering.
	  \item \label{p7} The countable atomless Boolean algebra with the canonical ordering.
	\end{enumerate}
\end{problem}
At the time~\cite{Kechris2005} was published, even the corresponding infinite Ramsey-theoretic results (let alone the topological dynamics analogs) were only fully understood for~\ref{p1} and~\ref{p6}. A Ramsey theorem for~\ref{p6} was given by Carlson~\cite{carlson1987infinitary} (see~\cite[Section~5.7]{todorcevic2010introduction}), and we will discuss case~\ref{p1} below, after introducing some notions of structural Ramsey theory.

Given structures $\str{A}$ and $\str{B}$, we denote by $\Emb(\str A,\str B)$ the set of
all embeddings from $\str{A}$ to $\str{B}$.
If $0< r,t < \omega$, we write $\str{B}\longrightarrow (\str{B})^\str{A}_{r,t}$ as a shortcut for the following statement:
\begin{quote}
	For every $r$-coloring $\chi\colon\Emb(\str{A},\str{B})\to r$, there exists an embedding
	$f\in \Emb(\str{B},\str{B})$ such that $|\{\chi(f\circ e):e\in \Emb(\str{A},\str{B})\}|\leq t$.
\end{quote}
For a countably infinite structure $\str{B}$ and its finite substructure
$\str{A}$, the \emph{big Ramsey degree} of $\str{A}$ in $\str{B}$ is the least
positive $\ell<\omega$, if it exists, such that $\str{B}\longrightarrow
(\str{B})^\str{A}_{r,\ell}$ for every positive $r<\omega$.
We say that \emph{the big
Ramsey degrees of $\str{B}$ are finite} if for every finite substructure
$\str{A}$ of $\str{B}$ the big Ramsey degree of $\str{A}$ in $\str{B}$ exists. 

We note the following two ways that the big Ramsey degree of $\str{A}$ in $\str{B}$ can fail to exist. Call a coloring $\chi$ of $\Emb(\str{A}, \str{B})$ \emph{unavoidable} if for every $f\in \Emb(\str{B}, \str{B})$, we have $\{\chi(f\circ e): e\in \Emb(\str{A}, \str{B})\} = \mathrm{Im}(\chi)$.
\begin{enumerate}
    \item 
    For every positive $r< \omega$, there is an unavoidable coloring $\chi\colon \Emb(\str{A}, \str{B})\to r$.
    \item 
    There is an unavoidable coloring $\chi\colon \Emb(\str{A}, \str{B})\to \omega$.
\end{enumerate}
Clearly $(2)$ implies $(1)$. We do not know of any examples where $(1)$ holds and $(2)$ fails. Indeed, letting $\str{B}_3$ and $\str{B}$ be the $3$-atom and countable atomless Boolean algebras, respectively, Theorem~\ref{thm:main} can be restated as asserting that there is an onto unavoidable coloring $\chi\colon \Emb(\str{B}_3, \str{B})\to \omega$. 

Refining an unpublished upper-bound theorem of Laver, Devlin gave the following precise infinite Ramsey-theoretic result
for the order of the rationals.
\begin{theorem}[Devlin, 1979~\cite{devlin1979}]
	\label{thm:devlin}
	Denote the order of the rationals by $(\mathbb Q,<)$.
	For every finite linear order $\str{A}$, the big Ramsey degree of $\str{A}$ in $(\mathbb Q,<)$ is 
	precisely the $|A|$-th \emph{odd tangent number}, namely the derivative of order $2|A|-1$ of $\tan(z)$ evaluated at $0$.
\end{theorem}
Problem~\ref{prob:1}, in particular for the unordered versions of each class therein, motivated research on similar characterisations of big
Ramsey degrees of various structures.  For a given countably infinite structure $\str{K}$, the problem can be attacked in three steps.
\begin{enumerate}
	\item Prove a structural Ramsey theorem giving upper bounds on big Ramsey degrees of finite substructures of $\str{K}$.
	\item Fully characterise the big Ramsey degrees of $\str{K}$ by proving a lower bound theorem 
		and refining the upper bound theorem so that they match.
	\item Give an interpretation of these results in the context of topological dynamics.
\end{enumerate}

Devlin's theorem thus completes the first two steps for the order of the rationals, case~\ref{p1}. Analogous results for case~\ref{p2}
appeared shortly after publication of~\cite{Kechris2005}.
The upper bound on big Ramsey degrees of the countable random (or Rado) graph can be
given by a natural adaptation of Laver's proof, which is today stated as an application of the Milliken tree theorem; see~\cite[Chapter 6]{todorcevic2010introduction} and~\cite{Sauer2006}.
A precise characterisation of big Ramsey degrees of the random graph was given in 2006 by Laflamme, Sauer, and Vuksanovic~\cite{Laflamme2006}.

Zucker~\cite{zucker2017} introduced
the notion of a \emph{big Ramsey structure}, a strengthening of having finite big Ramsey degrees which gives rise to a dynamical object called the \emph{universal completion flow} and gives
a solution to the last step of the attack on Problem~\ref{prob:1} above for the first three cases; see also the recent survey~\cite{hubicka2024survey}. However, it remains unclear what the correct analog of a big Ramsey structure should be when big Ramsey degrees are infinite, or whether the universal completion flow exists for every topological group.

For over a decade, further progress on Problem~\ref{prob:1} has been blocked by a lack of techniques for giving upper bounds on big Ramsey degrees
for structures with non-trivial forbidden configurations.
Recent rapid progress started with upper bounds on big Ramsey degrees of
the ``random'' (i.e.\ universal and homogeneous) triangle-free graph given
by Dobrinen~\cite{dobrinen2017universal}; see also~\cite{Hubicka2020CS} for a shorter proof.
Dobrinen subsequently extended the method to $\str{K}_k$-free
graphs~\cite{dobrinen2019ramsey}. Zucker further generalized and simplified the
proof for finitely constrained free amalgamation classes in finite binary
languages~\cite{zucker2020}. (We remark that the assumption about the finiteness of the language is necessary~\cite{Omegalabelled2025}, and there are known examples of infinitely constrained classes without finite big Ramsey degrees~\cite{sauer2003canonical}.) The precise characterisation of big Ramsey degrees of structures covered by Zucker's result was given by Balko, Chodounsk\'y,
Dobrinen, Hubi\v{c}ka, Kone\v{c}n\'y, Vena, and Zucker~\cite{Balko2021exact}, yielding big Ramsey structures for case~\ref{p3} (see also~\cite{Vodsedalek2025}).  These results are based on the new technique of coding trees with an upper bound theorem proved using the language of set-theoretic forcing; see Dobrinen's ICM survey~\cite{dobrinen2020forcing}.
More recently, a new method extending the Carlson--Simpson theorem has been introduced~\cite{Hubicka2020CS,Balko2023Sucessor}, which generalizes to
additional structures including partial orders and metric spaces.

Work on big Ramsey degrees of structures in languages with relations of arity 3 or more, considered in case~\ref{p4}, introduces new challenges.
An upper bound theorem for the random 3-uniform hypergraph (i.e.\ the case when $\mathcal A=\emptyset$)
was proved by Balko, Chodounsk{\'y}, Hubi{\v{c}}ka, Kone{\v{c}}n{\'y}, and Vena~\cite{Hubickabigramsey}, and 
this construction was subsequently generalized to structures in arbitrary relational languages and certain very simple forbidden configurations by Braunfeld, Chodounsk{\'y}, de Rancourt, Hubi{\v{c}}ka, Kawach, and Kone{\v{c}}n{\'y}~\cite{braunfeld2023big}.
So far, big Ramsey degrees and big Ramsey structures have not been fully described for case~\ref{p4}.

Case~\ref{p5} of Problem~\ref{prob:1} presents a different complication.
It is easy to show that the rational Urysohn space has infinite big Ramsey degrees of vertices~\cite{hubicka2024survey}. However, an approximate (metric) form of a
vertex coloring theorem, known as \emph{oscillation stability}, can be formulated and was proved for the Urysohn sphere. A combinatorial strategy for the proof
was given by Lopez-Abad and Nguyen Van Th{\'e}~\cite{lopez2008oscillation}, reducing the problem to a vertex-coloring big Ramsey problem for Urysohn spheres with finitely
many distances, which was proved by Sauer and Nguyen Van Th{\'e}~\cite{NVT2009b} (see also a recent easier proof~\cite{bice2025oscillation}).
Work on generalizing big Ramsey degrees to the metric setting is currently in progress; see the announcement~\cite{Bice2023}.

\medskip

Individual cases of Problem~\ref{prob:1} are motivated by the corresponding finitary
Ramsey-theoretic results. Case~\ref{p1} corresponds to the Ramsey theorem.
Case~\ref{p2} corresponds to the unrestricted form of the Ne\v set\v ril--R\"odl theorem~\cite{Nevsetvril1977}
(also independently proved by Abramson--Harrington~\cite{Abramson1978}). Cases~\ref{p3} and~\ref{p4}
correspond to the full strength of the Ne\v set\v ril--R\"odl theorem, and case~\ref{p5} corresponds to a result
of Ne\v set\v ril showing that ordered metric spaces form a Ramsey class~\cite{Nevsetvril2007}.
Case~\ref{p6} corresponds to a special case of the Graham--Leeb--Rothschild theorem~\cite{Graham1972},
and finally case~\ref{p7} corresponds to the dual Ramsey theorem (a special case of the Graham--Rothschild theorem~\cite{Graham1971}). See, e.g.,~\cite{hubicka2025twenty} for a discussion
of finite structural Ramsey theory. 

In this setting, it is natural to ask about big Ramsey degrees of the countable atomless Boolean algebra
and expect a ``dual form'' of Theorem~\ref{thm:devlin}.
It thus may come as a surprise that by Theorem~\ref{thm:main} the big Ramsey degree of the
3-atom Boolean algebra is infinite.
This means that case~\ref{p7} differs significantly from cases~\ref{p1},~\ref{p2}, and~\ref{p3}, as well as from the existing partial results on case~\ref{p4}
(it is possible that in case~\ref{p4} 
 big Ramsey degrees are not finite for many choices of $\mathcal A$~\cite{sauer2003canonical,typeamalg}). 
There are also significant differences from case~\ref{p5}.
Metric big Ramsey degrees of the Urysohn sphere can be seen as a limit of
approximations by Urysohn spheres with finitely many distances 
each having finite big Ramsey degrees~\cite{Hubicka2020CS,balko2021big}.
This is in contrast with the countable atomless Boolean algebra, which has
no natural approximations.
Finally, it is rather curious that while the finitary Ramsey-theoretic variants of cases~\ref{p6} and~\ref{p7} are closely connected (they follow by basically the same proof), the infinite settings are completely different. The Ramsey theorem for case~\ref{p6}, proved by Carlson~\cite{carlson1987infinitary} (see~\cite[Section~5.7]{todorcevic2010introduction}), says that the big Ramsey degrees of finite-dimensional subspaces are finite, in contrast with Theorem~\ref{thm:main}.

Consequently, we believe that solving case~\ref{p7} will require
development of novel tools.  We expect that the problem can still be addressed in
the following three steps.
\begin{problem}
	Determine the big Ramsey degree of the 2-atom Boolean algebra in the countable atomless Boolean algebra.
\end{problem}
\begin{problem}
	\label{probopt}
	Prove the corresponding infinite Ramsey-theoretic result for the countable atomless Boolean algebra and
	show its optimality even in the case where the big Ramsey degrees are infinite.
\end{problem}
\begin{problem}
	Generalize the notion of big Ramsey structure to this new setting, where big Ramsey degrees are infinite,
	and find its topological dynamics analogue.
\end{problem}
We remark that the solution to Problem~\ref{probopt} may take some inspiration from~\cite{Evans2}.

\section{An explicit representation of the countable atomless Boolean algebra}

We use the standard set-theoretic notation. In particular, we identify a nonnegative integer $n$ with the set $\{0,\ldots, n-1\}$ and $\omega$ with the set of nonnegative integers.

Put $X = [0, 1)\cap \mathbb{Q}$, and let $\prec$ denote the usual linear order on $X\cup\{1\}$. The \emph{interval algebra} of $X$ is the Boolean subalgebra of $\mathcal{P}(X)$ consisting of finite unions of half-open intervals; see \cite{BAHandbook}. More precisely, let us call $a\subseteq X$ \emph{good} if there exists $n<\omega$ and an increasing function $\seq_a\colon 2n\to X\cup \{1\}$ such that given $x\in X$, we have $x\in a$ if and only if for some $i<n$, we have $\seq_a(2i)\preceq x\prec \seq_a(2i+1)$. The function $\seq_a$ is uniquely determined by $a$. Note in particular that $\min(a) = \seq_a(0)$ for non-empty good $a\subseteq X$.
Denote by $\str{B}$ the set of all good subsets of $X$ viewed as a Boolean subalgebra of $\mathcal{P}(X)$.
\begin{observation}
	$\str B$ is a countable atomless Boolean algebra.
\end{observation}
\begin{proof}
    The set $\str B$ is closed under unions and intersections, the complement of every good set is good, and both $\emptyset$ and $X$ are good. So $\str B$ is a subalgebra of $\mathcal{P}(X)$.

    Since good subsets correspond to finite sequences of rationals, $\str B$ is countable. 
    For every nonempty $a\in \str B$ there is a nonempty good set $b \subsetneq a$, so $\str B$ is the unique countable atomless Boolean algebra up to isomorphism. 
\end{proof}

\section{Oscillation}

Given the representation $\str B$ of the countable atomless Boolean algebra, we are now ready to
introduce the coloring function which features in Theorem~\ref{thm:main}. 
We will first assign a finite set of natural numbers $\int(a)$ to each element $a$ of $\str B$, 
and the coloring function will be defined on pairs $(a, b)$ of elements of $\str B$ as the so-called oscillation
of the sets $\int(a)$ and $\int(b)$. 
To color the 3-atom subalgebras of $\str B$ we will use the oscillation coloring applied to their atoms. 
The idea of utilizing the oscillation coloring is inspired by a similar result 
on infinite big Ramsey degrees of the pseudotree 
(the universal countable homogeneous meet-tree) 
by Chodounský, Eskew, and Weinert~\cite{ChEW_pseudotree}. 
The concept of oscillations originates in the work of Todorcevic~\cite{Todorcevic_Roscillate}.

Fix a finite-to-1 function $e\colon X\cup \{1\}\to \omega$. Although a bijection would suffice for the present proof, the more general finite-to-1 formulation may be useful in future applications; the only property needed below is that the preimage of every finite subset of $\omega$ is finite.
Given $a\in \str{B}$, let
\[\int(a)=\{e(u):u\in \mathrm{Im}(\seq_a)\}\]
 be the set of ``interesting'' numbers associated with $a$.
\begin{definition}
	Given $a_0,a_1\in \str{B}$ and $i < \omega$, we say that $a_0$ and $a_1$ \emph{oscillate} at $i$ if there exists $k\in \{0,1\}$ such that:
    \begin{itemize}
        \item 
        $i\in \int(a_k)\setminus\int(a_{1-k})$,
        \item 
        if $i\cap (\int(a_k)\mathop{\Delta} \int(a_{1-k})) \neq \emptyset$ ($\Delta$ here denotes symmetric difference) and $j = \max\{i\cap (\int(a_k)\mathop{\Delta} \int(a_{1-k}))\}$, then $j\in \int(a_{1-k})\setminus \int(a_k)$. 
    \end{itemize}
	We let

	\[\osc(a_0,a_1) = \left|\left\{i < \omega : \text{$a_0$ and $a_1$ oscillate at $i$} \right\}\right|.\]
\end{definition}
Equivalently, $\osc(a_0,a_1)$ is the number of maximal blocks in the increasing enumeration of $\int(a_0)\mathop{\Delta}\int(a_1)$, where each point is labelled according to which of $\int(a_0)$ and $\int(a_1)$ contains it.
Notice that the function is symmetric, i.e.\ $\osc(a, b) = \osc(b, a)$ for all $a,b \in \str{B}$.
We prove Theorem~\ref{thm:main} in the following more technical form.
\begin{theorem}
	\label{thm:osc}
	For every countable atomless subalgebra $\str{C}$ of $\str{B}$ and every $n\geq 2$ there exist nonempty $a,b\in \str{C}$ such that $a\cap b = \emptyset$, $0 \notin a \cup b$, and $\osc(a,b) = n$.
\end{theorem}

\begin{proof}[Proof of Theorem~\ref{thm:main} using Theorem~\ref{thm:osc}]
    Let $\str A$ be the Boolean algebra with 3 atoms. 
    Define a coloring $\chi\colon \Emb(\str A, \str B) \to \omega$ as follows: 
    Given $f \in \Emb(\str A, \str B)$, let $a, b, c \in \str B$ be the images of the three atoms of $\str A$ under $f$ such that $\min(a) \prec \min(b) \prec \min(c)$, and put $\chi(f) = \max\{0,\osc(b,c)-2\}$. This value depends only on the image $f[\str A]$, so $\chi$ induces a coloring of the 3-atom subalgebras of $\str B$.

    We claim that every countable atomless subalgebra $\str C$ of $\str{B}$ realizes every color. 
    Indeed, fix any $n < \omega$ and use Theorem~\ref{thm:osc} to obtain $b,c\in \str{C}$ such that $b\cap c=\emptyset$, $\min X\notin b\cup c$ and $\osc(b,c) =n+2$. Let $\str D$ be the subalgebra of $\str C$ generated by $b$ and $c$. Its three atoms are $b$, $c$, and $X\setminus(b\cup c)$. The last of these contains $\min X$, and hence for every isomorphism $f\colon\str A\to \str D$ we have $\chi(f) = n$.
\end{proof}
We now turn to the proof of Theorem~\ref{thm:osc}.
\begin{lemma}
	\label{lem:oscN}
	Let $\str{C}$ be a countable atomless Boolean subalgebra of $\str{B}$. Then for every non-empty $a\in \str{C}$ and $n < \omega$ there exists a non-empty $b\in \str{C}$ such that $b\subsetneq a$ and $\min(\int(b))>n$.
\end{lemma}
\begin{proof}
	Fix $\str{C}$, $a$ and $n$.
	Put $L=\{x\in X\cup\{1\}:e(x)\leq n\}$.
	Because $\str{C}$ is atomless and $L$ finite, one can choose a set $A$ of $2|L|+1$ pairwise disjoint non-empty elements of $\str{C}$ such that every $a'\in A$ satisfies $a'\subseteq a$.
	Given $x\in L$ and $a'\in A$ we say that $x$ \emph{blocks} $a'$ if $x\in \mathrm{Im}(\seq_{a'})$.
	Notice that every $x\in L$ blocks at most two elements of $A$, since at most one member can have an interval beginning at $x$ and at most one can have an interval ending at $x$.
	Consequently, there exists $b\in A$ such that $\mathrm{Im}(\seq_b)\cap L=\emptyset$ and therefore $\min(\int(b))>n$. Since $A$ has more than one non-empty member, $b\subsetneq a$.
\end{proof}
\begin{lemma}
	\label{lem:oscP1}
	Let $\str{C}$ be a countable atomless Boolean subalgebra of $\str{B}$.
	For every non-empty $a,b\in \str{C}$ there exist non-empty $a',b'\in \str{C}$ such that $a'\subseteq a$, $b'\subseteq b$ and $\osc(a',b')=\osc(a,b)+1$.
\end{lemma}
\begin{proof}
	Fix $\str{C}$, $a$ and $b$.  Apply Lemma~\ref{lem:oscN} to obtain $c\in \str{C}$ satisfying $c\subsetneq a$ and $\min(\int(c))>\max(\int(b)\cup \int(a))$.
	Put $a'=a\setminus c$. For good sets $u,v$, the endpoints of $u\mathop{\Delta}v$ are precisely the symmetric difference of the endpoint sets of $u$ and $v$: the characteristic function changes value precisely at an endpoint. Since $c\subseteq a$ and $\int(a)\cap\int(c)=\emptyset$, the endpoint sets of $a$ and $c$ are disjoint. Thus the endpoint set of $a'$ is their disjoint union, and applying $e$ gives $\int(a')=\int(a)\mathbin{\dot\cup}\int(c)$. All elements of $\int(c)$ lie above $\int(a)\cup\int(b)$ and form one final block on the $a'$ side. Hence we have either $\osc(a',b)=\osc(a,b)$ or $\osc(a',b)=\osc(a,b)+1$. In the second
	case, take $b'=b$ and we are done.

	If $\osc(a',b)=\osc(a,b)$, then the last block of $\int(a')\mathop{\Delta}\int(b)$ is on the $a'$ side. Apply Lemma~\ref{lem:oscN} to obtain $d\in \str{C}$ satisfying $d\subsetneq b$ and $\min(\int(d))>\max(\int(a'))$. Put $b' = b\setminus d$. As above, $\int(b')=\int(b)\mathbin{\dot\cup}\int(d)$. Since $\max(\int(a'))>\max(\int(b))$, the set $\int(d)$ forms a new final block on the $b'$ side.
	We have $\osc(a',b')=\osc(a',b)+1=\osc(a,b)+1$.
\end{proof}
\begin{proof}[Proof of Theorem~\ref{thm:osc}]
	Fix an arbitrary countable atomless subalgebra $\str{C}$ of $\str{B}$.
	Choose a non-empty proper element $u\in\str{C}$ and let $v$ be whichever of $u$ and $X\setminus u$ does not contain $0$. By atomlessness, choose a non-empty proper $a\subsetneq v$ in $\str{C}$ and put $b=v\setminus a$. Thus $a$ and $b$ are non-empty disjoint elements of $\str{C}$ such that $0\notin a\cup b$.
	By Lemma~\ref{lem:oscN}, obtain $b'\in \str{C}$ such that $b'\subseteq b$ and $\min(\int(b'))>\max(\int(a))$.
	We have $\osc(a,b')=2$.
	By a repeated application of Lemma~\ref{lem:oscP1}, we can then obtain a pair of non-empty disjoint elements of $\str{C}$ with any given oscillation $n\geq 2$; these elements remain subsets of $a$ and $b'$ and hence avoid $0$.
\end{proof}

\section{Alternative proof via topological copies of the rationals}

We would like to point out that Theorem~\ref{thm:main} can also be derived from known results about colorings of pairs of rational numbers that are persistent on topological subcopies of the rationals. Let us sketch the argument. For simplicity, we will only reprove the analogue of Theorem~\ref{thm:main} for colorings of 4-atom subalgebras. Improving this proof to produce a coloring of 3-atom subalgebras requires a finer technical analysis of the argument.

We will utilize the following result:

\begin{theorem}[Baumgartner~\cite{baumgartner1986}, see also Theorem 6.31~\cite{todorcevic2010introduction}]
	\label{6.31}
	There is a coloring $c\colon \mathbb Q^{[2]} \to \omega$ that takes
	every value in $\omega$ on any set of the form $H^{[2]}$ for $H \subseteq \mathbb Q$ homeomorphic to $\mathbb Q$.
\end{theorem}

\begin{lemma}\label{lem:isol_points}
    Let $\mathbf C$ be an infinite Boolean algebra and let $i \colon \mathbf C \to \mathcal{P}(\omega)$ be an embedding into the Boolean algebra of subsets of the natural numbers.
    Then $i[\mathbf C]$ has no isolated points in the Cantor space topology of $\mathcal{P}(\omega)$.
\end{lemma}
\begin{proof}
    Let $a \in i[\mathbf C]$ and $n< \omega$. We need to find $b \in i[\mathbf C]$ such that $a \neq b$ and $a \cap n = b \cap n$.
    Since $i[\mathbf C]$ is infinite and there are only finitely many possible intersections with $n$, there are distinct $x_0, x_1 \in i[\mathbf C]$ such that $x_0 \cap n = x_1 \cap n$. Let $x = x_0 \mathop{\Delta} x_1 \in i[\mathbf C]$. Then $x\neq\emptyset$ and $x\cap n=\emptyset$, so we can let $b = a \mathop{\Delta} x$.
\end{proof}

The analogous argument for 4-atom subalgebras can now be sketched as follows.
Let $\mathbf B$ be the countable atomless Boolean algebra, and fix an embedding $i \colon \mathbf B \to \mathcal P(\omega)$.
The space $i[\mathbf B]$ is a countable subspace of the metrizable space $\mathcal P(\omega) \cong 2^\omega$ with no isolated points, and is thus homeomorphic to the space of rationals. By Theorem~\ref{6.31}, there is a persistent coloring $\chi \colon i[\mathbf B]^{[2]} \to \omega$; define $c \colon \mathbf B^{[2]} \to \omega$ by $c(\{a,b\}) = \chi(\{i(a), i(b)\})$.
Whenever $\mathbf C$ is an infinite subalgebra of $\mathbf B$, the space $i[\mathbf C]$ has no isolated points by Lemma~\ref{lem:isol_points}, and $c[\mathbf C^{[2]}] = \omega$.
That is, for every $n < \omega$ there is a pair of elements of $\mathbf C$ with color $n$. Since each pair of elements of an infinite Boolean algebra is contained in a 4-atom subalgebra, we can derive a coloring of 4-atom subalgebras of $\mathbf B$ which has infinitely many colors on every infinite subalgebra.
It is possible to obtain a coloring of 3-atom subalgebras in a similar way when we assume that $\mathbf C$ is atomless (not just infinite), by analyzing the proof of Theorem~\ref{6.31} to find a pair of elements of $\mathbf C$ with color $n$ which generate a 3-atom subalgebra of $\mathbf C$.

\section*{AI use statement}

The results of this paper were obtained and written without the use of any AI tools. Prior to journal submission, ChatGPT 5.6 was used to proofread the draft and catch some minor errors.

\section*{Acknowledgments}
D.B.\ was supported by grants DMS-1953955 and CAREER DMS-2144118 of the National Science Foundation (NSF).
D.Ch.\ and J.H.\ were supported by project 25-15571S of the  Czech Science Foundation (GA\v CR). 
D.Ch.\ was also partially supported by the Czech Academy of Sciences (RVO 67985840). B.C.\ was supported by an NSERC Discovery Grant.
J.H.\ was also supported by a project that has received funding from the European Research Council (ERC) under the European Union's Horizon 2020 research and innovation programme (grant agreement No 810115). 
M.K.\ was supported by a project that has received funding from the European Union (Project POCOCOP, ERC Synergy Grant 101071674).  Views and opinions expressed are however those of the authors only and do not necessarily reflect those of the European Union or the European Research Council Executive Agency. Neither the European Union nor the granting authority can be held responsible for them. A.Z.\ was supported by NSERC grants RGPIN-2023-03269 and DGECR-2023-00412.

\bibliographystyle{alpha}
\bibliography{ramsey.bib}
\end{document}